\documentclass{article}[12 pt, english]
\usepackage{amssymb}

\newtheorem{Zae}{Zae}[section] 
\newtheorem{definition}[Zae]{Definition} 
\newtheorem{lemma}[Zae]{Lemma} 
\newtheorem{prop}[Zae]{Proposition}
 
\newtheorem{thm}[Zae]{Theorem} 
\newtheorem{coro}[Zae]{Corollary}

\newcommand{\qed}{\raisebox{-.8ex}{$\Box$}}
\newenvironment{proof}%%
{\noindent{\bf Proof.}}% teil1
{\hfill \qed\\}% teil2
\newenvironment{proofofcoro}%%
{\noindent{\bf Proof of Corollary 1.3.}}% teil1
{\hfill \qed\\}% teil2
\newenvironment{proofof1.2}%%
{\noindent{\bf Proof of Theorem 1.2.}}% teil1
{\hfill \qed\\}% teil2
\newenvironment{proof1.4}%%
{\noindent{\bf Proof of Theorem 1.4.}}% teil1
{\hfill \qed\\}% teil2

\newcommand{\characteristic}{\hbox{char}}

\begin{document}
\title{Multiplicative quadratic maps}      
\author{Matthias Gr\"uninger\footnote{supported by the ERC (grant \# 278 469)}\\
Universit\'e catholique de Louvain \\
Institute pour recherche en math\'ematique et physique \\
Chemin du Cyclotron 10, bte. L7.01.01.\\
1348 Louvain-la-Neuve \\
Belgique}   
\maketitle    

\begin{abstract} 
In this paper we prove that a multiplicative quadratic map between a unital ring $K$ and a field $L$ is induced by a homomorphism from $K$ into $L$ or a composition algebra over
$L$. 
Especially we show that if $K$ is a field, then every multiplicative quadratic map is the product of two 
field homomorphisms. Moreover, we prove a multiplicative version of Artin's Theorem showing that 
a product of field homomorphisms is unique up to multiplicity.
\end{abstract}     
\date{27.07.2014}
\thanks{This work was completed with the support  by the ERC (grant \# 278 469)} 
\maketitle
\section{Introduction}
We begin with the following definition:
\begin{definition} 
Let $K$ and $L$ be (not necessarily associative) rings with $1$.
A map $q:K \to L$ is called a multiplicative quadratic map if
\begin{enumerate}\item $q(ab) =q(a) q(b)$ for all $a,b \in K$.
\item $q(n \cdot 1_K) =n^2\cdot 1_L$ for all $n \in \mathbb{Z}$.
\item The map $f: K \times K \to L$ defined by $f(a,b) =q(a+b)-q(a)-q(b)$ is biadditive.
\end{enumerate}
\end{definition}
If $M$ is a composition algebra over $L$ with norm $N$ and $\varphi_:K \to M$ a non-zero
homomorphism, then $q:K \to L: a \mapsto N(a^{\varphi})$ is a multiplicative quadratic 
map. If $\characteristic L=2$, then every non-zero homomorphism from $K$ to $L$ is a 
multiplicative quadratic map.   
 We will show that if $L$ is a field, then these are in fact all multiplicative quadratic 
maps between $K$ and $L$. 
\begin{thm}\label{qmaps} Let $K$ be a unital ring with $1$ and $L$ a field. 
If ${q: K \to L}$ is a 
multiplicative quadratic map, then one of the following holds:
\begin{enumerate}
\item There is a composition algebra $M$ over $L$ and a homomorphism\\ 
${\varphi_:K \to M}$ with $q(a) =N(a^{\varphi})$ for all $a \in K$, where $N$ is the norm of $M$. 
\item $\characteristic L =2$ and $q$ is a ring homomorphism.
\end{enumerate}
\end{thm}
For arbitrary rings the situation is more complicated, especially in even characteristic.
 For example, let $K =\mathbb{Z}/2\mathbb{Z}$ and $L =\mathbb{Z}/4\mathbb{Z}$. Then 
the map $q:K \to L$ defined by $q(1) =1$ and $q(0) =0$ is a multiplicative quadratic map. 
Another example for a muliplicative quadratic map is the adjoint map $\#: K \to K^{op}$ 
with $K$ a cubic algebra and $K^{op}$ its opposite algebra. 
\\
We are mainly interested in the case that $K$ is a field as well. Here we get
\begin{coro}\label{homom} 
If $K$ and $L$ are fields, $\overline{L}$ the algebraic closure of $L$ and $q:K \to L$ is a multiplicative quadratic map, then there are monomorphisms $\varphi_1,\varphi_2: K \to \overline{L}$ with $q(a) =a^{\varphi_1} a^{\varphi_2}$.
\end{coro}
The uniqueness of $\varphi_1$ and $\varphi_2$ in \ref{homom} can be easily deduced
from Artin's Theorem. Our second main theorem, which can 
can be seen as a multiplicative version of Artin's Theorem, is a generalization of this 
fact.
\begin{thm}\label{uniqueness of products}
Let $K,L$ be fields, $n \geq m \geq 1$ natural numbers and $\sigma_1,\ldots, \sigma_n,$
$ \tau_1, \ldots, \tau_m: K \to L$ 
field homomorphisms. Suppose that $\prod_{i=1}^n a^{\sigma_i} = \prod_{j=1}^m a^{\tau_j}$ for all $a \in K$. 
Then one of the following holds:
\begin{enumerate}
\item $n=m$ and there is a permutation $g \in S_n$ with $\tau_i=\sigma_{ig}$ for $1 \leq i \leq n$.
\item $char K= char L =p >0$ and there are $1 \leq i_1 < i_2  < \ldots < i_p \leq n$ with 
$\sigma_{i_1} =\sigma_{i_2}= \ldots =\sigma_{i_p}$.
Moreover, for all $1 \leq j \leq m$ 
there is an integer $l$ with $-(m-1) \leq  l(p-1) \leq n-1$ and a natural number 
$1 \leq i \leq n$ with $\tau_j =\sigma_i \mathbf{p}^l $ and
for all $1 \leq i \leq n$ there is an integer $l$ with $-(n-1) \leq l(p-1) \leq m-1$ and $1 \leq j \leq m$ 
with $\sigma_i =\tau_j \mathbf{p}^l $. Here $\mathbf{p}$ denotes the Frobenius endomorphism of $L$.  
\end{enumerate}
\end{thm}
For $n=m=2$ we get $\{\sigma_1,\sigma_2\}=\{\tau_1,\tau_2\}$ or $char K =char L =2$ and 
$\sigma_1 =\sigma_2$ and $\tau_1=\tau_2$. Hence $\sigma_1 \mathbf{p} =\tau_1\mathbf{p} $ and so $\sigma_1 =\tau_1$.
\\
As another corollary we get
\begin{coro} Let $K$ be a field and $S$ a set of field endomorphisms of $K$. Suppose that 
$\characteristic K =0$ or $\characteristic K=p$ and that if $\sigma, \tau \in S$ and $n \in \mathbb{N}$ with 
$\sigma =\mathbf{p}^n \tau$, then $\sigma =\tau$. Then $S$ is $\mathbb{Z}$-linear independent in $End(K^*)$. 
\end{coro}
The initial motivation for these questions lays in the theory of Moufang sets and its connection to (twin) trees.
(See \cite{CDM} for an introduction into this topic.)
Suppose that $T$ is a tree and $G$ a subgroup of $Aut T$ such that the following hold:
\begin{enumerate}
\item For every vertex $x$ of $T$ there is a field $K_x$ such that $G_x$ induces a group isomorphic to $PSL_2(K_x)$ on the 
neighbourhood of $x$.
\item $G$ induces a Moufang set on the set of ends of $T$, i.e. for every end $e$ the group $G_e$ has a normal subgroup 
$U_e$ which acts regularly on the set of ends distinct from $e$. 
\end{enumerate}
 Then under certain conditions we get multiplicative quadratic maps between $K_x$ and $K_y$ for all vertices $x$ and $y$.
 The classification of all mulitiplicative quadratic maps helps us to classify these trees. We refer to two 
 forthcoming papers (\cite{CG},\cite{G}) which will use both of our main theorems. 
\section{Multiplicative quadratic maps}
In the following $K$ and $L$ are rings and 
$q:K \to L$ is a multiplicative quadratic map with associated biadditive map $f$.
\begin{lemma}\label{formulas} Let $a,b,c,d \in K$. Then
\begin{enumerate} \item $f(ac,bc) =f(a,b)q(c)$ and $f(ca,cb) =q(c) f(a,b)$.
\item $f(a,b)f(c,d) =f(ac,bd) + f(ad,bc)$.
\item If $I$ is a right 
ideal of $K$, then $I^{\perp}:=\{a \in K;  f(a,b) =0\ \forall b \in I \}$ is also a right ideal of $K$.
The same holds for left ideals of $K$.
\item $rad(q):=\{a \in K^{\perp}; q(a) = $ is an ideal of $K$.
\item $\overline{q}:K/rad(q) \to L: a+rad(q) \mapsto q(a)$ is a well-defined multiplicative quadratic map with 
$rad(\overline{q}) =0$.
\end{enumerate}
\end{lemma}
\begin{proof}
\begin{enumerate} \item We have
$$f(ac,bc) = q(ac+bc) -q(ac)-q(bc) =(q(a+b)-q(a)-q(b))q(c) = f(a,b)q(c).$$
The second equation follows by a similar computation.
\item By 1. we have $$q(a+b)q(c+d) =(q(a) +q(b)+f(a,b))(q(c) +q(d) +f(c,d)) = $$ 
$$q(ac) +q(ad) +f(ac,ad)+q(bc)+q(bd) +f(bc,bd) +f(ac,bc)$$
$$ +f(ad,bd) +f(a,b)f(c,d).$$
On the other hand, 
$$q(a+b)q(c+d) = q(ac +ad +bc+bd) = $$
$$q(ac) +q(ad)+q(bc) +q(bd) + f(ac,ad)+f(ac,bc)+f(ac,bd)$$
$$+f(ad,bc)+f(ad,bd)+f(bc,bd).$$
If we compare these two equations, we get
$$f(a,b)f(c,d) =f(ac,bd) +f(ad,bc).$$
\item 
Suppose that $I$ is a right ideal of $K$, $b \in I$ and $a \in I^{\perp}$. If we set $d=1$ in 2. we get
$$0=f(a,b)f(c,1) = f(ac,b) +f(a,bc) =f(ac,b).$$
This shows that $ac \in I^{\perp}$. Since $f$ is biadditive, $I^{\perp}$ is additively closed. Thus $I^{\perp}$ is a
right ideal of $K$. For left ideals of $K$ the proof is similar. 
\item Let $a,b \in rad(q)$ and $c,d \in K$. Then $f(a+b,c) =f(a,c) +f(b,c) =0$ and $q(a+b) =q(a) +q(b) +f(a,b) =0$, 
thus $a+b \in rad(q)$. Moreover, $a \in K^{\perp}$, which is an ideal of $K$, so also $ ac \in K^{\perp}$ and thus 
$f(ac,d) =0$. Since $q(ac) =q(a) q(c) =0$, we get $ac \in rad(q)$.
\item For all $a \in K$ and $b \in rad(q)$ we get $q(a+b) =q(a) +q(b) +f(a,b) =q(a)$, so $\overline{q}$ is 
well-defined. The rest is clear.
\end{enumerate}
\end{proof}
We will say that $q$ is {\it non-degenerate} if $rad(q)=\{0\}$. 
\begin{lemma}\label{qhomomorphism} Suppose that $a \in K$ with $f(a,b) =0$ for all $b\in K$. Then 
\begin{enumerate}
\item $q(a) f(b,c)=f(b,c) q(a) =0$ for all $b,c \in K$.
\item If $q(a)$ is not a zero-divisor, then $\characteristic L =2$ and $q$ is a ring homomorphism.
\end{enumerate}
\end{lemma}
\begin{proof}
By assumption $a \in K^{\perp}$, which is an ideal by \ref{formulas}. Thus for all $b,c \in K$ we have
$ab,ac,ba,ca \in K^{\perp}$ and so $0 = f(ab,ac) =q(a) f(b,c) $ and
$0 =f(ba,ca) =f(b,c) q(a)$. Thus 1. follows. If $q(a) $ is not a zero-divisor, 
then $f$ is identically zero, thus $q$ is a ring homomorphism. Since $4 
\cdot 1_L = q(2 \cdot 1_K) = 2 \cdot 1_L$ holds in $L$, we get 
$\characteristic L =2$.
\end{proof}
If $F$ is a commutative, associative ring such that $K$ and $L$ are $F$-algebras, 
we say that $q$ is $F$-quadratic if $f$ is $F$-bilinear and $q(\lambda a) =\lambda^2 q(a)$ for all $a \in K$ and 
$\lambda\in F$. Note that $q$ is always $\mathbb{Z}$-quadratic. If $K$ and $L$ are fields, 
then they have 
the same characteristic (since $q(n \cdot 1_K) =n^2 \cdot 1_L$ for all $n \in \mathbb{N}$), and one can easily 
see that 
$q$ is $F$-quadratic with $F=\mathbb{Q}$ or $F=\mathbb{F}_p$ for a prime $p$. \\
From now on $L$ is associative and commutative. Suppose that $F$ is a subring of $L$ such that 
$K$ is a $F$-algebra and $q$ is $F$-quadratic. Moreover, suppose that every finitely generated submodule of  
$K$ is a free $F$-module. Set $\tilde{K}= L \otimes_F K$. Then $\tilde{K}$ is a $L$-module.   
For $a \in K$ and $\lambda \in L$ we write $\lambda a$ instead of $\lambda \otimes a$.
We define ${\tilde q}: {\tilde K} \to L$ by 
$${\tilde q}(\sum_{i=1 }^n \lambda_i a_i) = \sum_{i=1}^n  \lambda_i^2 q(a_i) + \sum_{1 \leq i <j \leq n} 
\lambda_i \lambda_j f(a_i,a_j)$$
for $a_1, \ldots, a_n \in K$ and $\lambda_1, \ldots, \lambda_n \in L$
and set $\tilde{f}(a,b) =\tilde{q}(a+b)-\tilde{q}(a)-\tilde{q}(b)$ for $a,b \in \tilde{K}$.
Then
\begin{prop}\label{tensor}
${\tilde q}$ is a multiplicative $L$-quadratic map.
\end{prop}
\begin{proof} 
Since every finitely generated $F$-module of $K$ is free, one can prove by a standard argument that 
$\tilde{q}$ is well-defined.
 Let $a_1, \ldots, a_n \in K, \lambda_1, \ldots, \lambda_n,$ \\ $ \mu_1, \ldots,
  \mu_n \in L$ and  
 $x =\sum_{i=1}^n  \lambda_i a_i, y=\sum_{j=1}^n \mu_j a_j \in {\tilde K}$.
Then $${\tilde f}(x,y) = 
\tilde{q}(\sum_{i=1}^n (\lambda_i +\mu_i) a_i) - 
\tilde{q}(\sum_{i=1}^n \lambda_i a_i) -\tilde{q}(\sum_{i=1}^n \mu_i a_i) = $$
$$\sum_{i=1}^n (\lambda_i+\mu_i)^2 q(a_i) +\sum_{i<j} (\lambda_i +\mu_i)(\lambda_j +\mu_j) f(a_i,a_j) -$$
$$\sum_{i=1}^n \lambda_i^2 q(a_i) -\sum_{i<j} \lambda_i \lambda_j f(a_i,a_j)
 -\sum_{i=1}^n \mu_i^2 q(a_i) -\sum_{i<j} \mu_i \mu_j f(a_i,a_j)=$$
$$\sum_{i=1}^n 2\lambda_i \mu_i q(a_i) +\sum_{ i < j}  (\lambda_i \mu_j+\mu_i \lambda_j) f(a_i,a_j)=$$
$$\sum_{i=1}^n \lambda_i \mu_i f(a_i,a_i) +\sum_{ i<j } (\lambda_i \mu_j+\lambda_j \mu_i) f(a_i,a_j) =
\sum_{i,j=1}^n \lambda_i \mu_j f(a_i,a_j).$$
Thus ${\tilde f}$ is $L$-bilinear. By definition $\tilde{q}(\lambda a) = \lambda^2 \tilde{q}(a)$ for 
all $a \in \tilde{K}$ and all $\lambda \in L$.
It remains to show that ${\tilde q}$ is multiplicative. 
We define an ordering $\subset$ on $\{1, \ldots, n\} \times \{1, \ldots, n\}$ by 
$(i,j) \subset (k,l)$ iff $i < k$ or $i=k$ and $j <l$.
Using \ref{formulas} we get 
$${\tilde q}(x) {\tilde q}(y) =
(\sum_{i=1}^n \lambda_i^2 q(a_i) + \sum_{ i <k } 
\lambda_i \lambda_k f(a_i,a_k))(\sum_{j=1}^n \mu_j^2 q(a_j) 
 + \sum_{j <l } \mu_j 
\mu_l f(a_j,a_l))=$$
$$\sum_{i,j=1}^n \lambda_i^ 2 \mu_j^2 q(a_i)q( a_j) + \sum_{i=1}^n \sum_{ j <l }
 \lambda_i^2 \mu_j \mu_l q(a_i)f( a_j, a_l) + $$
$$\sum_{i <k } \sum_{j=1}^n \lambda_i \lambda_k \mu_j^2 f(a_i ,a_k )q(a_j) +
\sum_{i <k } \sum_{ j<l } \lambda_i \lambda_k \mu_j \mu_l f(a_i,a_k) 
f(a_j,a_l)=$$
$$\sum_{i,j=1 }^n \lambda_i^2 \mu_j^2 q(a_i a_j) + \sum_{i=1} \sum_{ j <l } \lambda_i^2 \mu_j \mu_l f(a_i a_j,a_i a_l) +
\sum_{i <k } \sum_{j=1}^n \lambda_i \lambda_k \mu_j^2 f(a_i a_j,a_k a_j) $$
$$+\sum_{i <k } \sum_{j<l } \lambda_i \lambda_k \mu_j \mu_l (f(a_i a_j, a_k a_l)+ 
f(a_i a_l, a_ka_j))=$$
$$\sum_{i,j=1}^n \lambda_i^2 \mu_j^2 q(a_i a_j) + \sum_{i=1}^n \sum_{ j <l } 
\lambda_i^2 \mu_j \mu_l 
f(a_i a_j,a_i a_l) +
\sum_{ i <k } \sum_{j=1}^n \lambda_i \lambda_k \mu_j^2 f(a_i a_j,a_k a_j)+ $$
$$\sum_{i <k  } \sum_{j \ne l} \lambda_i \lambda_k \mu_j \mu_l f(a_i a_j ,a_k a_l)=
\sum_{i,j=1}^n \lambda_i^2 \mu_j^2 q(a_i a_j) +$$
 $$\sum_{(i,j) \subset (k,l)}  \lambda_i \lambda_k \mu_j \mu_l f(a_i a_j ,
a_k a_l).$$
But we have 
$$xy= \sum_{i,j=1}^n  \lambda_i \mu_j a_i a_j $$ and thus 
$${\tilde q}(xy)  =\sum_{i,j=1}^n \lambda_i^2 \mu_j^2 q(a_i a_j) + \sum_{(i,j) \subset (k,l)} 
\lambda_i \mu_j \lambda_k \mu_l f(a_i a_j,a_k a_l) ={\tilde q}(x) {\tilde q}(y).$$
\end{proof}
\bigskip \\
\begin{proofof1.2}
By \ref{formulas}.5. we may assume that $q$ is non-degenerate. 
If $char L =p >0$, then $q(p\cdot 1_K) =p^2 \cdot 1_L=0$ and $f(p \cdot 1_K, a) =
p f(1_K,a)=0$ for all 
$a \in K$, so 
$char K =p$ as well.
If $char L =0$, then $q(n \cdot 1_K) =n^2 \cdot
 1_L \ne 0$ for all natural numbers $n \geq 0$, so 
$char K =0$ as well. So $K$ and $L$ have the same characteristic.    
Set $F=\mathbb{Z}$ if $\characteristic K =\characteristic L =0$ and $F=\mathbb{F}_p$ if $\characteristic K =\characteristic L 
=p$.
 Using \ref{tensor}, we get a multiplicative $L$-quadratic map 
 $\tilde{q}: L \otimes_F K \to L$. Thus 
$\tilde{q}$ is a quadratic form of the $L$-vectorspace $L \otimes_F K$. Set $M:= L \otimes_F K 
/rad(\tilde{q})$ and let $N:M \to L: N(a+rad(\tilde{q}))=\tilde{q}(a)$.
Then by \ref{formulas}.5. $N$ is well-defined. 
Since $rad(N) =0$, by \ref{qhomomorphism}.2. either $\characteristic L =2$ and 
$N$ is a ring homomorphism or the bilinear form associated to $N$ 
is non-degenerate. In the first case $q$ is a homomorphism as well. 
Since $f$ is identically zero and $rad(q)=\{0\}$, $q$ must be injective. 
\\
In the second case $M$ is a composition algebra over $L$ with norm $N$.
 Set $\varphi:K \to M: x \mapsto 1 \otimes x + rad(N)$, Since $rad(q)=0$, $\varphi$ is an 
embedding of $K$ in $M$ with $N(x^{\varphi}) =q(x)$ for all $x \in K$.
\end{proofof1.2}
\bigskip \\
\begin{proofofcoro}
If $K$ is a field, then in the first case of \ref{qmaps} $M$ must be a commutative 
composition algebra. Thus by 1.6.2 of \cite{SV} we have either $M =L$ 
and $N(x) =x^2$ for $x \in M$ 
or $M \cong L \times L $ and $N(x_1,x_2)=x_1 x_2$ for $(x_1,x_2) \in M$ or $M$ is a separable 
quadratic extension of $L$ and $N(x)=x x^{\sigma}$ for $x \in M$ with $\sigma $ the non-
trivial 
Galois automorphism of $M$. In the first case we have $q(x) =(x^{\varphi})^2$ for an embedding 
$\varphi_: K \to L$, in the second case there are embeddings $\varphi_1,\varphi_2:K \to L$ 
with $q(x) =x^{\varphi_1} x^{\varphi_2}$, in the last case there is an embedding 
$\varphi:K \to M$ with $q(x)=x^{\varphi} x^{\varphi \sigma}$ for $x \in K$.  
\\
In the second case of \ref{qmaps} we have $q(a) =a^{\varphi} a^{\varphi}$ with $\overline{L}$
 the algebraic closure of $L$
and $\varphi_: K \to \overline{L}:a \mapsto \sqrt{q(a)}$. 
\end{proofofcoro}

\section{The uniqueness of products of field homomorphisms}
We will now deal with the uniqueness of the representations of maps of the form 
$$x \mapsto \prod_{i=1}^n x^{\sigma_i}$$
with $\sigma_1, \ldots, \sigma_n:K \to L$ field homomorphisms. In general, such a representation is not unique. 
For example, let $L=K$ be a field of characteristic $2$, let $\mathbf{p}$ be the Frobenius endomorphism and 
set $\sigma_1 =\sigma_2 =\sigma_3:= \mathbf{p}$, $\tau_1 =\mathbf{p}^2$ and $\tau_2=\tau_3 =id$. Then we have 
$\prod_{i=1}^3 x^{\sigma_i} =x^6 =\prod_{i=1}^3 x^{\tau_i}$. 
\begin{lemma}\label{3.1} Suppose that $K$ and $L$ are fields and $\sigma_1,\ldots, \sigma_n:K \to L $ are monomorphisms. 
If ${\sum_{g \in S_n} \prod_{i=1}^ n x_i^{\sigma_{ig}} =0}$ for all $x_1, \ldots, x_n \in K$, then 
$\characteristic K = p > 0$ and there are $1 \leq i_1 < \ldots < i_p  \leq n$ with $\sigma_{i_1} = \ldots = \sigma_{i_p}$. 
\end{lemma}
\begin{proof} We have 
$$\sum_{g \in S_n} \prod_{i=1}^n x_i^ {\sigma_{ig}} = \sum_{i=1}^n (\sum_{g \in S_n, ng =i} \prod_{j=1}^{n-1} 
x_j^ {\sigma_{jg}}) x_n^{\sigma_i}.$$
Set $a_i(x_i,\ldots, x_{n-1}):=\sum_{g \in S_n,ng =i} \prod_{j =1}^{n-1} x_j^{\sigma_{jg}}$. Then 
$$\sum_{g \in S_n} \prod_{i=1}^n x_i^{\sigma_{ig}} =\sum_{i=1}^n a_i(x_1, \ldots, x_{n-1}) x_n^{\sigma_i}.$$
If $a_n(x_1, \ldots, x_{n-1})=0$ for all 
$x_1, \ldots, x_{n-1} \in K$, then the claim follows by induction. 
Suppose that there are $x_1, \ldots, x_{n-1}$ such that $a_n(x_1, \ldots, x_{n-1}) \ne 0$. 
We may assume that there is $0 \leq m < n$ with 
$\sigma_i \ne \sigma_n$ for $i \leq m$ and  
$\sigma_i= \sigma_n$ for $i > m$. Then 
$$0=  \sum_{g \in S_n} \prod_{i=1}^n x_i^ {\sigma_{ig}} =\sum_{i=1}^{m} a_i(x_1, \ldots, 
x_{n-1}) x_n^{\sigma_i} +$$
$$ (n-m) a_n(x_1, \ldots, x_{n-1}) x_n^{\sigma_n}.$$
By the Theorem of Artin (\cite{A}, III 1, Satz 13) this implies $$(n-m)
a_n(x_1, \ldots, x_{n-1}) =0,$$ thus $\characteristic K = p >0$ and $p$ divides $n-m$. 
\end{proof}
\bigskip\\
\begin{proof1.4}
We prove the claim by induction on $n+m$. 
For $x_1, \ldots, x_n \in K$, $\emptyset \ne J \subseteq \{1, \ldots, n\}$ set $x_J =\sum_{i \in J} x_j$ and 
$$f(x_1, \ldots, x_n) := \sum_{\emptyset \ne J \subseteq \{1, \ldots, n\}} (-1)^{|J|} \prod_{i=1}^n x_J^{\sigma_i}
= \sum_{\emptyset \ne J \subseteq \{1, \ldots, n\}} (-1)^{|J|} \prod_{i=1}^m x_J^{\tau_i}.$$
Then 
$$f(x_1, \ldots, x_n) = \sum_{g \in S_n} \prod_{i=1}^n x_i^ {\sigma_{ig}}=\sum_{g \in S_n} \prod_{i=1}^n x_i^ {\tau_{ig}}$$
if $m=n$ and 
$$f(x_1, \ldots, x_n) = \sum_{g \in S_n} \prod_{i=1}^n x_i^ {\sigma_{ig}}=0$$
if $m<n$. 
Again, set $a_i(x_1, \ldots, x_{n-1}) = \sum_{g \in S_n, ng =i} \prod_{j=1}^{n-1} x_j^{\sigma_{ig}}$. If $m=n$, we set 
$b_i(x_1, \ldots, x_{n-1}) = \sum_{g \in S_n, ng =i} \prod_{j=1}^{n-1} x_j^{\tau_{ig}}$. 
If $m<n$, we set $b_i(x_1, \ldots, x_{n-1})$ $ =0$. In both cases we have  
$$\sum_{i=1}^n a_i (x_1, \ldots, x_{n-1}) x_n^{\sigma_i} = f(x_1, \ldots, x_n) = \sum_{i=1}^m b_i(x_1, \ldots, x_{n-1})x_n^{\tau_i}.$$
 First suppose $\characteristic K =0$ or $\characteristic K =p$ and there are no $p$ identical $\sigma_i$, then by \ref{3.1} there are 
 $x_1, \ldots, x_n$ such that $f(x_1, \ldots, x_n) \ne 0$. Again by Artin (\cite{A}, III 1, Satz 13) we get $m=n$ and there are $i,j$ with 
 $\sigma_i =\tau_j$. We may assume $i=j=n$. Then we have 
$\prod_{i=1}^{n-1} x^{\sigma_i} = \prod_{i=1}^{n-1} x^{\tau_i}$ for all $x \in K$ and thus 1. holds by induction.  
\\
Now suppose that $char K =char L =p$ and that
$\sigma_n = \sigma_{n-1} = \ldots =\sigma_{n-p+1}$. Then we have
$$a^{\sigma_1}  \ldots  a^{\sigma_{n-p} }a^{\sigma_n \mathbf{p}} =a^{\tau_1}  \ldots  a^{\tau_m} $$
for all $a \in K$ with $\mathbf{p}$ the Frobenius endomorphism. Hence for all $1 \leq j \leq m$  
there is a $l \in \mathbb{Z}$ with 
$-m+1 \leq l(p-1) \leq n-p$ such that there is an $1 \leq i \leq n-p$ with 
$\tau_j =\sigma_i \mathbf{p}^l$ or $\tau_j =\sigma_n \mathbf{p}^{l+1}$. 
In the last case we replace $l$ by $l+1$ and get that there is an integer $l$ with
$-m+1 \leq l(p-1)\leq n-1$ and $1 \leq i \leq n$ with $ \tau_j = \sigma_i \mathbf{p}^l$.
\\
We have to prove that for $1 \leq i \leq n$ there is an integer $l$ with $-n+1 \leq l(p-1) \leq m-1$ and 
$1 \leq j \leq m$ with $\sigma_i = \tau_j \mathbf{p}^l$. This follows directly for $i \leq n-p$. 
Moreover there is an integer $l$ with $-n+p \leq l(p-1) \leq m-1$ and 
$1 \leq j \leq m$ with $\sigma_n \mathbf{p}= \tau_j \mathbf{p}^l$ and hence $\sigma_n =\tau_j \mathbf{p}^{l-1}$. 
If we replace $l$ by $l-1$, the claim follows.  
\end{proof1.4}

\end{document}